\documentclass[a4paper,12pt]{article}
\usepackage{comment,amsmath,amssymb,amsthm,pstricks,changepage,pifont,graphicx}
\usepackage[english]{babel}

\newtheorem{theorem}{Theorem}

\newtheorem{lemma}[theorem]{Lemma}

\newtheorem{question}[theorem]{Question}

\newenvironment{remark}[1][Remark]{\begin{trivlist}
\addtocounter{theorem}{1} \item[\hskip \labelsep {\itshape #1 \thetheorem}.]} {\end{trivlist}}

\newenvironment{example}[1][Example]{\begin{trivlist}
\addtocounter{theorem}{1} \item[\hskip \labelsep {\itshape #1 \thetheorem}.]} {\end{trivlist}}

\makeatletter
\newcommand{\manuallabel}[2]{\def\@currentlabel{#2}\label{#1}}
\makeatother

\DeclareMathOperator{\conv}{conv}

\newcommand*{\Scale}[2][4]{\scalebox{#1}{$#2$}}%

\begin{document}

\title{A combinatorial interpretation for Schreyer's tetragonal invariants}
\author{Wouter Castryck\footnote{Supported by FWO-Vlaanderen.}\phantom{ } and Filip Cools}
\date{}

\maketitle

\begin{abstract}
  \noindent Schreyer has proved that the graded Betti numbers of a
  canonical tetragonal curve are determined by two integers
  $b_1$ and $b_2$, associated to the curve through a certain geometric construction.
  In this article we prove that in the case
  of a smooth projective tetragonal curve on a toric surface, these integers
  have easy interpretations in terms of the Newton polygon of its
  defining Laurent polynomial. We can use this to prove an intrinsicness
  result on Newton polygons of small lattice width.\\

\end{abstract}

\section{Introduction} \label{introsection}

Let $k$ be an algebraically closed field of characteristic $0$ and let $\mathbb{T}^2 = (k^\ast)^2$
be the two-dimensional torus over $k$. Let $\Delta \subset \mathbb{R}^2$
be a two-dimensional lattice polygon and consider the associated toric surface
$\text{Tor}(\Delta)$ over $k$, i.e.\ the Zariski closure of the image of
\[ \varphi_\Delta : \mathbb{T}^2 \hookrightarrow \mathbb{P}^{\sharp (\Delta \cap \mathbb{Z}^2) - 1} : (\alpha,\beta) \mapsto
 (\alpha^i \beta^j)_{(i,j) \in \Delta \cap \mathbb{Z}^2}.
\]
Let
\[ f = \sum_{(i,j) \in \mathbb{Z}^2} c_{i,j} x^iy^j \in k[x^{\pm 1}, y^{\pm 1}] \]
be an irreducible Laurent polynomial and consider its Newton polygon
\[ \Delta(f) = \conv \left\{ \, \left. (i,j) \in \mathbb{Z}^2 \, \right| \, c_{i,j} \neq 0 \, \right\}. \]
Let $U_f \subset \mathbb{T}^2$ be the curve cut out by $f$.
We say that $f$ is $\Delta$-non-degenerate
if $\Delta(f) \subset \Delta$ and
for every face
$\tau \subset \Delta$ (vertex, edge, or $\Delta$ itself) the system
\[ f_\tau = \frac{\partial f_\tau}{\partial x} = \frac{\partial f_\tau}{\partial y} = 0 \]
has no solutions in $\mathbb{T}^2$.
Here
\[ f_\tau = \sum_{(i,j) \in \tau \cap \mathbb{Z}^2} c_{i,j} x^i y^j.\]
For a fixed instance of $\Delta$ and given that $\Delta(f) \subset \Delta$, the condition of $\Delta$-non-degeneracy is generically satisfied.
It implies that the
Zariski closure $C_f$ of $\varphi_\Delta(U_f)$ inside $\text{Tor}(\Delta)$ is non-singular.
A curve that is isomorphic to $C_f$ for some $\Delta$-non-degenerate Laurent polynomial
is in turn called $\Delta$-non-degenerate.\\

\noindent Non-degenerate curves form an attractive class of objects from the point
of view of explicit algebraic geometry.
On the one hand they vastly generalize
well-known families such as elliptic curves, hyperelliptic curves, trigonal curves\footnote{Strictly spoken, there do exist
trigonal curves that are not non-degenerate; for example see \cite[Lem.\,4.4]{CaCo1}. But all trigonal curves are `morally' non-degenerate, in the sense
that they can always be embedded in a toric surface, which is sufficient for most applications. See also the remark at the end of this section.}, smooth plane curves,
$C_{a,b}$ curves, \dots \ covering a much broader range of geometric situations.
On the other hand they remain very tangible, because many important geometric invariants
can be
told by simply looking at the combinatorics of $\Delta$. Two notable instances are:
\begin{itemize}
 \item the (geometric) \emph{genus} $g$, which equals $\sharp (\Delta^{(1)} \cap \mathbb{Z}^2)$, where $\Delta^{(1)}$ is
 the convex hull of the interior lattice points of $\Delta$; see \cite{Khovanskii};
 \item the \emph{gonality} $\gamma$, which equals $\text{lw}(\Delta)$, except if $\Delta \cong 2\Upsilon$ or
 $\Delta \cong d\Sigma$ for some $d \geq 2$, where
 \[ \Upsilon = \conv \{ (-1,-1), (1,0), (0,1) \} \qquad \text{and} \qquad \Sigma = \conv \{ (0,0), (1,0), (0,1) \}, \]
 in which case it equals $\text{lw}(\Delta) - 1$;
 here $\text{lw}$ denotes the lattice width, and $\cong$ indicates unimodular equivalence; see \cite[Lem.\,6.2]{CaCo1}.
 (Shorter characterization: $\gamma = \text{lw}(\Delta^{(1)}) + 2$ except if $\Delta \cong 2\Upsilon$ in which case $\gamma = 3$.)
\end{itemize}
Similar interpretations exist for the \emph{Clifford index} and the \emph{Clifford dimension} \cite[\S8]{CaCo1},
and in some cases for the \emph{minimal degree of a plane model} \cite{latticesize}. The current paper extends the list of combinatorial
features of non-degenerate curves, by focusing on tetragonal curves.
Namely, we give the following
interpretation for the invariants $b_1$ and $b_2$, as introduced by Schreyer
in \cite[(6.2)]{Schreyer}.
The definition of these invariants will be recalled in Section~\ref{sect_schreyersinv} below.

\begin{theorem} \label{schreyerinvinl}
Let $C$ be a tetragonal $\Delta$-non-degenerate curve.
Then Schreyer's corresponding set of invariants $\{ b_1,b_2 \}$ is given by
\[ \left\{ \ \sharp (\partial \Delta^{(1)} \cap \mathbb{Z}^2) - 4 \, , \ \sharp (\Delta^{(2)} \cap \mathbb{Z}^2) - 1 \ \right\}. \]
\end{theorem}

\noindent Here $\partial$ denotes the boundary and $\Delta^{(2)} = \Delta^{(1)(1)}$ is
the convex hull of the interior lattice points of $\Delta^{(1)}$.

\begin{example} The Laurent polynomial $f = 1 + y^2 - x^6y^2 + x^6y^4 \in \mathbb{C}[x,y]$ is $\Delta$-non-degenerate,
where $\Delta$ is as follows.
\begin{center}
  \includegraphics[height=2.5cm]{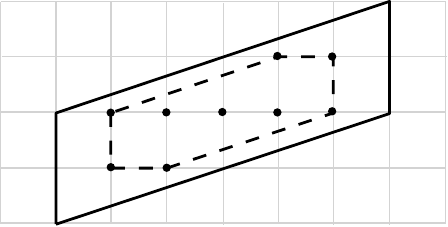}
\end{center}
The dashed lines indicate $\Delta^{(1)}$.
One verifies, purely by looking at the Newton polygon, that $C_f$ is a tetragonal curve of genus $9$ with $b_1 = b_2 = 2$.
(In view of \cite[Cor.\,6.3, Thm.\,9.1]{CaCo1}, one can even say that it carries a unique $g^1_4$, whose scrollar invariants
read $1,1,4$; see Remark~\ref{scrollardef} below for more background on this terminology.)
\end{example}

\noindent Schreyer's invariants are known to determine the \emph{Betti diagram of the canonical ideal}, and vice versa \cite[(6.2)]{Schreyer}.
In particular, Theorem~\ref{schreyerinvinl} implies that in the tetragonal case, the Betti diagram is 
combinatorially determined. We believe
that this holds in much greater generality (work in progress).\\

\noindent A second aim of this paper is to initiate a discussion on the \emph{intrinsicness} of $\Delta$.
Namely, given the many
geometric invariants that are encoded in the Newton polygon, one might wonder
to what extent it is possible to reconstruct $\Delta$ from the abstract geometry of
a given $\Delta$-non-degenerate curve $C_f$. The best one can hope for is to
find back $\Delta$ up to unimodular equivalence,
\begin{comment}
transformations of the form
\[ \mathbb{R}^2 \rightarrow \mathbb{R}^2 : \begin{pmatrix} x \\ y \\ \end{pmatrix} \mapsto A \begin{pmatrix} x \\ y \\ \end{pmatrix} + B \]
where $A \in \text{GL}_2(\mathbb{Z})$ and $B \in \mathbb{Z}^2$
(two lattice polygons $\Delta, \Delta'$ that are obtained from one another using such a transformation will be called \emph{equivalent},
denoted $\Delta \cong \Delta'$).
\end{comment}
because unimodular transformations correspond to
automorphisms of $\mathbb{T}^2$.
Another relaxation is that (usually) one can only expect to
recover $\Delta^{(1)}$, rather than all of
$\Delta$. For example, let
$f \in k[x, y]$ be $d\Sigma$-non-degenerate for some integer $d \geq 2$ and
let $(x_0,y_0) \in U_f$ be sufficiently
generic.
Then $f' = f(x + x_0, y+y_0)$ is $\Delta$-non-degenerate, where $\Delta$ is obtained
from $d\Sigma$ by clipping off the point $(0,0)$.
In this case $\Delta \not \cong d\Sigma$, while clearly
$C_f \cong C_{f'}$.
More generally, pruning a vertex off a lattice polygon $\Delta$ without affecting its interior
boils down to forcing the curve through a certain non-singular point of $\text{Tor}(\Delta)$, which
is usually not intrinsic.
One is naturally led to the following question.

\begin{question}[intrinsicness] \label{intrinsicquestion}
 Let $\Delta, \Delta'$ be two-dimensional lattice polygons 
 for which there exists a curve that is both $\Delta$-non-degenerate and $\Delta'$-non-degenerate.
 Does it follow that $\Delta^{(1)} \cong \Delta'^{(1)}$?
\end{question}

%(See the introduction of~\cite{intrinsicness_forthcoming} for a more extended discussion.)
\noindent Our conjecture is
that for `most' pairs of polygons the answer is yes. E.g., this is known to be true as soon as
\begin{itemize}
 \item[(a)] $\Delta^{(1)}$ is one-dimensional, because a $\Delta$-non-degenerate curve is hyperelliptic
 of genus $g \geq 2$ if and only if $\Delta^{(1)} \cap \mathbb{Z}^2$ consists of $g$ collinear points \cite[Lem.\,3.2.9]{Koelman},
 \item[(b)] $\Delta^{(1)} = \emptyset$ or
$\Delta^{(1)} \cong (d-3)\Sigma$ for
some integer $d \geq 3$, because a $\Delta$-non-degenerate curve is abstractly
isomorphic to a smooth plane curve if and only if $\Delta^{(1)}$ is a multiple of the standard simplex (up to equivalence) \cite[Cor.\,8.2]{CaCo1}.
 \item[(c)] $\Delta^{(1)} \cong [0,a] \times [0, b]$ for some integers $a \geq b \geq -1$ with $(a+1)(b+1) \neq 4$, 
 because a $\Delta$-non-degenerate curve of genus $g \neq 4$ can be embedded in $\mathbb{P}^1 \times \mathbb{P}^1$
 if and only if $\Delta^{(1)}$ is a standard rectangle (up to equivalence); see\footnote{This is an article in preparation.
 In the meantime, a proof can be found in \S5.10 of our unpublished
 \texttt{arXiv} paper \texttt{1304.4997}.} \cite{hirzebruch}.
 The assumption $g \neq 4$ is necessary: see the discussion following (d) below.
\end{itemize}
Let us indicate why we expect Question~\ref{intrinsicquestion}
to have an affirmative answer for many more instances of $\Delta$, while gathering some material that will be needed in
Section~\ref{sect_schreyersinv}. Our starting point is a theorem by Khovanskii \cite{Khovanskii}, stating
that there exists a canonical divisor $K_\Delta$ on $C_f$ such that a basis for
the Riemann-Roch space $H^0(C_f, K_\Delta)$ is given by
\begin{equation} \label{canemb}
\left\{ x^iy^j \right\}_{(i,j) \in \Delta^{(1)} \cap \mathbb{Z}^2}.
\end{equation}
Here $x,y$ are to be viewed as functions on $C_f$ through $\varphi_\Delta$. Note that one recovers
the statements that $g = \sharp (\Delta^{(1)} \cap \mathbb{Z}^2)$ and
that $C_f$ is hyperelliptic if and only if $\Delta^{(1)}$ is one-dimensional; see \cite[Lem.\,5.1]{castryckvoight} for more
details. If $\Delta^{(1)}$ is two-dimensional,
then Khovanskii's theorem implies
that the
canonical model $C_f^\text{can}$ of $C_f$ satisfies
\[ C_f^\text{can} \subset \text{Tor}(\Delta^{(1)}) \subset \mathbb{P}^{g-1}. \]
But surfaces of the form $\text{Tor}(\Delta^{(1)})$ are very special. Most notably,
they are of low degree, and they are generated by binomials. The idea is that
they are so special that there is room for at most one such surface containing
$C^\text{can}_f$. This idea is not always true, but the exceptions seem rare.
If it \emph{is} true, then the following general and seemingly new statement
allows one to recover $\Delta^{(1)}$. A proof will be given in Section~\ref{section_from_surface_to_polygon}.

\begin{theorem} \label{lem_equiv}
Let $\Delta, \Delta'$ be two-dimensional lattice polygons with
\[ \sharp (\Delta \cap \mathbb{Z}^2) - 1 =
\sharp (\Delta' \cap \mathbb{Z}^2) - 1 = N,\]
and suppose that $\emph{Tor}(\Delta) , \emph{Tor}(\Delta') \subset \mathbb{P}^N$
can be obtained from one another using a projective transformation.
Then $\Delta \cong \Delta'$.
\end{theorem}

\noindent Using this, we can immediately extend the above list to the case where
\begin{itemize}
 \item[(d)] $\sharp (\Delta^{(1)} \cap \mathbb{Z}^2) \geq 5$ and $\Delta^{(2)} = \emptyset$, which holds if and only if
 $C_f$ is trigonal of genus $g \geq 5$, or isomorphic to a smooth plane quintic \cite[\S8]{CaCo1}.
In this case $\text{Tor}(\Delta^{(1)})$ can be characterized as the unique irreducible surface containing
$C_f^\text{can}$ that is generated by quadrics. Indeed, the fact that it is
generated by quadrics follows from \cite{Koelman2}, while uniqueness follows from Petri's theorem \cite{saintdonat}.
\end{itemize}
The above argument breaks down in the genus $4$ case where $\Delta \cong 2\Upsilon$, because $\text{Tor}((2\Upsilon)^{(1)}) = \text{Tor}(\Upsilon)$ is not generated by quadrics.
And indeed, using this, it is not hard to cook up $2\Upsilon$-non-degenerate
curves that are also $[0,3] \times [0,3]$-non-degenerate, and $2\Upsilon$-non-degenerate curves
that are also $\conv \{(0,0), (4,0), (0,2)\}$-non-degenerate. (See \S5.6 of our unpublished
\texttt{arXiv} paper \texttt{1304.4997} for an extended discussion; see also Example~\ref{example0mod4} below.)\\

\noindent In Section~\ref{sect_schreyersinv} we will give a similar but more complicated recipe
for recovering $\text{Tor}(\Delta^{(1)})$ in most tetragonal cases. More precisely,
we extend the list with the situation where
\begin{itemize}
 \item[(e)] $\text{lw}(\Delta^{(1)}) = 2$ and $\sharp (\partial \Delta^{(1)} \cap \mathbb{Z}^2) \geq \sharp (\Delta^{(2)} \cap \mathbb{Z}^2) + 5$,
 which holds if and only if $C_f$ is tetragonal and $b_1 \geq b_2 + 2$. In this case $\text{Tor}(\Delta^{(1)})$
 can be characterized as the unique surface containing $C_f^\text{can}$ that is linearly
 equivalent to $2H - b_1R$, when viewed as a divisor inside
 the scroll spanned by a $g^1_4$.
\end{itemize}
More explanation will be given in Section~\ref{intrinsicness}.
Of course, in establishing this, we will make extensive use of Theorem~\ref{schreyerinvinl} and its proof.

\begin{comment}

\noindent Using Theorem~\ref{schreyerinvinl} and its proof, we can state an intrinsicness result on
the Newton polygon for non-degenerate curves of low gonality.

\begin{theorem} \label{intrinsicness_main}
 Let $C / k$ be a curve of genus $g$ and gonality $\gamma \leq 4$. Suppose that
 $C$ is at the same time $\Delta_1$-non-degenerate and $\Delta_2$-non-degenerate for lattice polygons
 $\Delta_1, \Delta_2$ that are not contained in the explicit list of exceptional polygons given in Section~\ref{intrinsicness} (the list is empty
 if $\gamma \leq 3$ or $g \equiv 2, 3 \bmod 4$). Then $\Delta_1^{(1)} \cong \Delta_2^{(1)}$.
\end{theorem}

\noindent Here by $\cong$ we mean that the polygons are unimodularly equivalent, i.e.\ they can be obtained
from one another using an affine transformation with linear part in $\text{GL}_2(\mathbb{Z})$
and translation part in $\mathbb{Z}^2$.\\

\end{comment}

\begin{remark}
Even though we formulate our results in terms of non-degenerate curves,
they remain valid for the slightly more general class of
\emph{arbitrary} smooth curves in toric surfaces.
Indeed, to a smooth (non-torus-invariant) curve $C$
in a toric surface $\varphi : \mathbb{T}^2 \hookrightarrow X$ one can always associate a
`defining Laurent polynomial' $f \in k[x^{\pm 1}, y^{\pm 1}]$, by which we mean a
generator of the ideal
of $\varphi^{-1}C$. It is well-defined up to multiplication by $cx^iy^j$ for some
$c \in k^\ast$ and $(i,j) \in \mathbb{Z}^2$.
One then just proceeds with $f$ and $\Delta = \Delta(f)$, as if $f$ were $\Delta$-non-degenerate.
We refer to \cite[\S4]{CaCo1}
for a more extended discussion.
\end{remark}

\section{Schreyer's tetragonal invariants} \label{sect_schreyersinv}

Let $C / k$ be a tetragonal curve of genus $g \geq 5$ and assume it to be canonically embedded
in $\mathbb{P}^{g-1}$. Fix a gonality pencil $g^1_4$ on $C$ and consider
$$S=\bigcup_{D\in g^1_4}\,\langle D\rangle\subset \mathbb{P}^{g-1},$$ where $\langle D\rangle\subset \mathbb{P}^{g-1}$
denotes the linear span of $D$.
One can show that $S$ is a rational normal threefold scroll whose type we denote by $(e_1,e_2,e_3)$, where
we assume $0\leq e_1\leq e_2\leq e_3$.
One has $\deg S=e_1+e_2+e_3=g-3$, and
$S$ is non-singular if and only if $e_1 > 0$. If $e_1 = 0$ then
the singularities are resolved by the natural map $\mu : \mathbb{P}(\mathcal{E}) \rightarrow S$, where $\mathcal{E}$ is the
locally free sheaf $\mathcal{O}(e_1)\oplus\mathcal{O}(e_2)\oplus\mathcal{O}(e_3)$ on $\mathbb{P}^1$;
if $e_1 > 0$ then $\mu$ is an isomorphism. The Picard group of $\mathbb{P}(\mathcal{E})$ is freely generated by the hyperplane class $H=[\mu^*(\mathcal{O}(1))]$ and the ruling class $R$ consisting of the fibers of the projection $\pi : \mathbb{P}(\mathcal{E}) \rightarrow \mathbb{P}^1$. The following intersection-theoretic identities hold: $H^3=g-3$,
$H^2\cdot R=1$ and $R\cdot R=0$. For more general background and references, see \cite[\S9]{CaCo1} and \cite[\S2-4]{Schreyer}.

\begin{remark} \label{scrollardef}
The numbers $e_1,e_2,e_3$ are called the \emph{scrollar invariants} of $C$ with respect to our $g_4^1$.
\end{remark}

\noindent Now let $C'$ be the strict transform under $\mu$ of our canonical curve $C\subset S$.
Schreyer proved that $C'$ is the complete intersection of surfaces $Y$ and $Z$ in $\mathbb{P}(\mathcal{E})$, with $Y\sim 2H-b_1R$, $Z\sim 2H-b_2R$, $b_1+b_2=g-5$ and $-1\leq b_2\leq b_1\leq g-4$.
He moreover showed
that $b_1,b_2$ are invariants of the curve: they depend neither on the canonical embedding, nor on the choice of the
$g^1_4$, nor on the choice of $Y$ and $Z$.
If $b_1 > b_2$, which is automatic if $g$ is even, then $Y$ is in fact unique,
 and $\mu(Y) \subset \mathbb{P}^{g-1}$ is
independent of the chosen $g^1_4$.
%(although it does of course depend on the canonical embedding).
For these particular statements we refer to \cite[(6.2)]{Schreyer}.\\

\noindent The main goal of this section is to prove the combinatorial interpretation for
Schreyer's invariants $b_1,b_2$ stated in Theorem~\ref{schreyerinvinl}.
Using the abbreviations
\[
B = \sharp(\partial \Delta^{(1)}\cap\mathbb{Z}^2)-4, \qquad
B^{(1)} = \sharp(\Delta^{(2)}\cap\mathbb{Z}^2)-1, \]
we will in fact show:

\begin{theorem} \label{theorem_schreyer}
Let $f \in k[x^{\pm 1}, y^{\pm 1}]$ be non-degenerate with respect to its Newton polygon
$\Delta = \Delta(f)$, and suppose that $C_f$ is tetragonal.
Then its invariants $b_1,b_2$ statisfy $\{b_1,b_2\}=\{B,B^{(1)}\}$.
If moreover $B > B^{(1)}$ then the surface $\mu(Y)$ associated to
the canonical model $C_f^\emph{can}$ from Section~\ref{introsection} equals $\emph{Tor}(\Delta^{(1)})$.
\end{theorem}

\begin{proof}
The assumption
that $C_f$ is tetragonal is equivalent to $\text{lw}(\Delta^{(1)}) = 2$ and $\Delta \not \cong 2\Upsilon$.
We can also suppose that $\Delta \not \cong 5\Sigma$, because this case can be reduced to
$$\Delta \cong \conv \{ (1,0),(5,0),(0,5),(0,1) \}$$
by means of a coordinate transformation, as explained in the discussion preceding Question~\ref{intrinsicquestion}. By \cite[Lem.\,5.2]{CaCo1} we can therefore suppose
that
\[ \Delta^{(1)} \subset \left\{ (X,Y)\in\mathbb{R}^2\,|\,0\leq Y\leq 2 \right\} \quad \text{and} \quad \Delta\subset \left\{ (X,Y)\in\mathbb{R}^2\,|\,-1\leq Y\leq 3 \right\}. \]
Then the projection map $U_f\to\mathbb{T}^1:(x,y)\mapsto x$ has degree $4$, i.e.\ it gives rise to a $g_4^1$ on $C_f$.
As remarked in Section~\ref{introsection}, the canonical model $C_f^\text{can}$ obtained using the basis (\ref{canemb}) of $H^0(C_f, K_\Delta)$
satisfies
\[ C_f^\text{can} \subset \text{Tor}(\Delta^{(1)}) \subset \mathbb{P}^{g-1}. \]
%If $\sharp(\Delta^{(2)}\cap\mathbb{Z}^2)\geq 2$, there is only one gonality pencil $g_4^1$ on $C$.
% and that $C\subset \text{Tor}(\Delta^{(1)})\subset \mathbb{P}^{g-1}$ is the canonical model of $C_f$ defined as above.
The scroll $S$ corresponding to our $g_4^1$ is easily seen to be the Zariski closure of the image of the map
$$\mathbb{T}^3\hookrightarrow \mathbb{P}^{g-1}:(\alpha,\beta, \gamma)\mapsto \left( (\alpha^i)_{(i,0)\in\Delta^{(1)}\cap\mathbb{Z}^2} : (\beta  \alpha^i)_{(i,1)\in\Delta^{(1)}\cap\mathbb{Z}^2} : (\gamma \alpha^i)_{(i,2)\in\Delta^{(1)}\cap\mathbb{Z}^2} \right).$$
(Note that the scrollar invariants $e_1,e_2,e_3$ are precisely the numbers $$\sharp\{(i',j')\in\Delta^{(1)}\cap\mathbb{Z}^2\,|\,j'=j\}-1$$ for $j=0,1,2$, up to order; for a generalization of this observation, see \cite[\S9]{CaCo1}.)
Moreover, one verifies that $S$ contains $\text{Tor}(\Delta^{(1)})$, i.e.\ the above chain of inclusions extends to
\[ C_f^\text{can} \subset \text{Tor}(\Delta^{(1)}) \subset S \subset \mathbb{P}^{g-1}. \]
%The restriction of $\text{Tor}(\Delta^{(1)})$ to the ruling $\mathbb{P}^2$ of $S$ above $\alpha \in\mathbb{T}^1$ has $\beta=y$ and $
%\gamma=y^2$ as parameter equations (on $\mathbb{T}^2\subset\mathbb{P}^2$), so it is a conic.

\noindent Now let $\mu : \mathbb{P}(\mathcal{E}) \rightarrow S$ be as above and denote by $C'$ the strict transform
of $C_f^\text{can}$ under $\mu$. Similarly, denote by $T'$ the strict transform of $\text{Tor}(\Delta^{(1)})$.
Write the divisor
class of $T'$ as $aH+bR$ with $a,b\in\mathbb{Z}$.
Let $F$ be the fiber of $\pi$ above $\alpha \in \mathbb{T}^1 \subset \mathbb{P}^1$. Then $\mu(F)$ is a $\mathbb{P}^2$ whose
intersection with $\text{Tor}(\Delta^{(1)})$ has $\beta = y$ and $\gamma = y^2$ as parameter equations
on $\mathbb{T}^2 \subset \mathbb{P}^2$. In particular this intersection is a conic, so
we have that
$$a=(aH+bR)\cdot H\cdot R=T'\cdot H\cdot R=2.$$
Next, we compute the intersection product $T' \cdot H^2$ in two ways. On the one hand
we find the degree of $\text{Tor}(\Delta^{(1)})$, which equals
$2\text{Vol}(\Delta^{(1)})$ because the Hilbert polynomial of $\text{Tor}(\Delta^{(1)})$ equals
the Ehrhart polynomial of $\Delta^{(1)}$, see \cite[Prop.\,9.4.3]{coxlittleschenck}. On the other hand one has $$T'\cdot H^2=(2H+bR)\cdot H^2=2(g-3)+b.$$
We obtain that $b=2\text{Vol}(\Delta^{(1)})-2(g-3) = -B$, where the latter equality follows from Pick's theorem.
In conclusion, $T' \sim 2H - BR$.
Now 
 \begin{itemize}
   \item if $Y = T'$ then it is immediate that $b_1=B$ and, consequently, $b_2=B^{(1)}$,
   \item if $Y \neq T'$ then if we intersect $Y\sim 2H-b_1R$ and $T'\sim 2H-BR$ on $\mathbb{P}(\mathcal{E})$, we obtain a (possibly reducible) curve
whose image under $\mu$ has degree $$H\cdot (2H-BR)\cdot (2H-b_1R)=4(g-3)-2b_1-2B\leq 4(g-3)-2(g-5)=2g-2.$$
This follows from $2b_1\geq b_1+b_2=g-5$ and $2B\geq B+B^{(1)}=g-5$ if $B\geq B^{(1)}$, and from $2b_1\geq b_1+b_2+1=g-4$ and $2B=g-6$ if $B<B^{(1)}$; see Lemma~\ref{lemma_comb} below.
In both cases, if either one of the inequalities would be strict, then we would run into a contradiction
because $C'$ is contained in
this intersection (and $\mu(C') = C_f^\text{can}$, being a canonical curve, has degree $2g-2$). We conclude that
$b_1=b_2=B=B^{(1)}=\frac{g-5}{2}$ or $b_1=B^{(1)}=\frac{g-4}{2}$ and $b_2=B=\frac{g-6}{2}$.
\end{itemize}
All conclusions follow.
\end{proof}

%\begin{remark} The invariants $b_1$ and $b_2$ completely determine the graded Betti numbers
%of $C$ (and conversely) \cite[\S6.2]{Schreyer}. %See Section~\ref{section_graded_betti} for some related comments.
%\end{remark}

\begin{remark}
Assume that $C_f$ is not isomorphic to a smooth plane quintic, i.e.\ $\Delta^{(1)}\not\cong 2\Sigma$. Then
by Petri's theorem \cite{saintdonat} the ideal of $C_f^\text{can}$ is generated by quadrics. In this case we can
construct (instances of) Schreyer's surfaces $Y, Z \subset \mathbb{P}(\mathcal{E})$ in a concrete way,
by explicitly giving the defining equations of $\mu(Y),\mu(Z) \subset S$.
Indeed, by \cite[Thm.\,4]{CaCo2} the ideal of $C_f^\text{can}$ is minimally generated by quadrics $$b_1,\ldots,b_r,b'_1\ldots,b'_s,\mathcal{F}_{2,w_1},\ldots,\mathcal{F}_{2,w_t},$$ where
\begin{itemize}
\item the $r={g-3\choose 2}$ binomials $b_i$ generate $\mathcal{I}(S)$,
\item the $s=(4g-6)-\sharp(2\Delta^{(1)}\cap \mathbb{Z}^2)$ binomials $b'_i$ cut $\text{Tor}(\Delta^{(1)})$ out in $S$,
\item $t=\sharp(\Delta^{(2)}\cap\mathbb{Z}^2) = B^{(1)} + 1$ and the quadrics
$\mathcal{F}_{2,w_i}$ are constructed in the explicit manner described in \cite{CaCo2}. Note that there is some freedom
in the way these quadrics arise.
\end{itemize}
Then if $\mathcal{F}_f \subset \mathbb{P}(\mathcal{E})$ denotes the strict transform under $\mu$ of the joint zero locus of
the quadrics $\mathcal{F}_{2,w_i}$, one can verify that $\mathcal{F}_f \sim 2H-B^{(1)}R$, so that one can take $Y =T'$ and $Z = \mathcal{F}_f$ if $B\geq B^{(1)}$, and
$Y=\mathcal{F}_f$ and $Z=T'$ if $B<B^{(1)}$.
%Consider the strict transform $U'\subset \mathbb{P}(\mathcal{E})$ under $\mu$ of the zero locus of the quadrics $F_{2,w}$ with $w\in \Delta^{(2)}$. Since $U'\sim 2H-B'R$, we can take $Y=T'$ and $Z=U'$ iff $B\geq B'$, or $Y=U'$ and $Z=T'$ iff $B<B'$.
%If $\Delta^{(1)}\not\cong 2\Sigma$, take the strict transform under $\mu$ of the zero locus of the three cubics $F_{3,w},F_{3,w'},F_{3,w''}$ (see Theorem \label{thmtetragonal}). ...
%(\textcolor[rgb]{1,0,0}{TO DO}: prove that the strict transform is equivalent to $2H-B'R$ in $\mathbb{P}(\mathcal{E})$)
\manuallabel{rmk_eqYZ}{\thetheorem}
\end{remark}

\noindent We end this section by explicitly listing the
lattice polygons for which $B\leq B^{(1)}$. 
%The classification for $B < B^{(1)}$ was used in the proof
%of Theorem~\ref{theorem_schreyer} above, while the case $B = B^{(1)}$ will play a role in Section~\ref{intrinsicness} below.
We will need the following property of two-dimensional lattice polygons of the form $\Delta^{(1)}$.
An edge $\tau$ of a two-dimensional lattice polygon $\Gamma$ is
always supported on a line $a_\tau X + b_\tau Y = c_\tau$ with $a_\tau, b_\tau, c_\tau \in \mathbb{Z}$ and
$a_\tau, b_\tau$ coprime. When signs are chosen appropriately, we can
assume that $\Gamma$ is contained in the half-plane $a_\tau X + b_\tau Y \leq c_\tau$.
Then the line $a_\tau X + b_\tau Y = c_\tau + 1$ is called the \emph{outward shift} of $\tau$. It is denoted
by $\tau^{(-1)}$, and the polygon (which may take vertices outside $\mathbb{Z}^2$)
that arises as the intersection of the half-planes $a_\tau X + b_\tau Y \leq c_\tau + 1$
is denoted by $\Gamma^{(-1)}$.
If $\Gamma = \Delta^{(1)}$ for some lattice polygon $\Delta$, then the outward shifts
of two adjacent edges of $\Gamma$ always intersect in a lattice point, and in fact $\Gamma^{(-1)} = \Delta^{(1)(-1)}$
is a lattice polygon. Moreover, $\Delta \subset \Delta^{(1)(-1)}$, i.e.\ $\Delta^{(1)(-1)}$ is the maximal
lattice polygon with respect
to inclusion for which the convex hull of the interior lattice points equals $\Delta^{(1)}$.
See \cite[\S4]{HaaseSchicho} or \cite[\S2.2]{Koelman} for proofs.\\
%We will usually denote $\Delta^{(1)(-1)}$ by $\Delta^\text{max}$.

\noindent Even though the following statement is purely combinatorial, given its geometric interpretation, it is natural to abbreviate $g = \sharp (\Delta^{(1)} \cap \mathbb{Z}^2)$.
Similarly, we will write $g^{(1)} = \sharp (\Delta^{(2)} \cap \mathbb{Z}^2)$.

\begin{lemma} \label{lemma_comb}
Let $\Delta$ be a lattice polygon with $\emph{lw}(\Delta^{(1)}) = 2$.
Then we have:
\begin{itemize}
\item $B<B^{(1)}$ if and only if
$$\Delta^{(1)}\cong\Gamma_{4k+4}:=\conv\,\{(0,0),(k,0),(2k+2,1),(k+1,2),(1,2)\}$$
for some integer $k\geq 0$. In this case $g=4k+4$, $B=2k-1$ and $B^{(1)}=2k$.
\item $B=B^{(1)}$ if and only if either
$$\Delta^{(1)}\cong\Gamma_{4k+5}^m:=\conv\,\{(0,0),(k,0),(2k+2,1),(k+m,2),(m,2),(0,1))\}$$
for some integers $k\geq 0$ and $0\leq m\leq k+2$ (in these cases, $g=4k+5$ and $B=B^{(1)}=2k$), or
$$\Delta^{(1)}\cong\Gamma_{4k+3}:=\conv\,\{(0,0),(k,0),(2k+1,1),(k+1,2),(1,2)\}$$
for some integer $k\geq 1$ (in this case, $g=4k+3$ and $B=B^{(1)}=2k-1$), or
$$\Delta^{(1)}\cong\Gamma_{4k+1}:=\conv\,\{(0,0),(k,0),(2k,1),(k,2),(1,2)\}$$
for some integer $k\geq 2$ (in this case, $g=4k+1$ and $B=B^{(1)}=2k-2$).
\end{itemize}
\end{lemma}

\begin{proof}
First we consider the polygons with $g^{(1)}$ equal to $0$ and $1$ separately.
If $g^{(1)}=0$ then $\Delta^{(1)}\cong 2\Sigma$, hence $B=2>B^{(1)}=-1$. If $g^{(1)}=1$ then $B^{(1)}=0$, hence $B\leq B^{(1)}$ if and
only if $g\leq 5$. It is easy to check that there is one such polygon in genus 4
(namely $\Delta\cong 2\Upsilon$, so $\Delta^{(1)}\cong \Upsilon=\Gamma_4$) and
three such polygons in genus $5$ (corresponding to $\Delta^{(1)} \cong \Gamma_5^0, \Gamma_5^1, \Gamma_5^2$). Each of these
appear in the classification.

If $g^{(1)}\geq 2$, we can use Koelman's classification \cite[Section 4.3]{Koelman} of lattice polygons $\Gamma$ with
lattice width $2$.
One can assume that $\Gamma=\Delta^{(1)}$ is contained
in the strip $\{(X,Y)\in\mathbb{R}^2\,|\,0\leq Y\leq 2\}$.
Koelman subdivided these polygons into three types:

\begin{itemize}
\item \emph{Type $0$: there is no boundary lattice point of $\Gamma$ with $Y=1$.} \\
Then up to equivalence $\Gamma=\Delta^{(1)}$ is of the form
\begin{center}
\includegraphics[height=2.5cm]{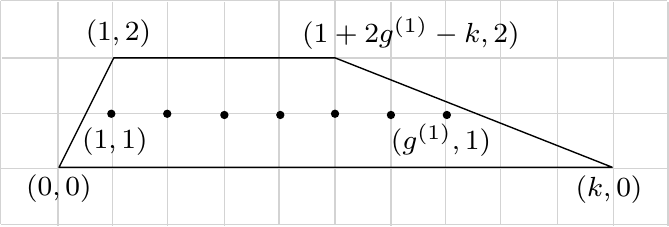}
\end{center}
with $g^{(1)}\leq k\leq 2g^{(1)}$. One sees that $B=2g^{(1)}-2$ and $B^{(1)}=g^{(1)}-1$, so $B\leq B^{(1)}$ implies that $g^{(1)}\leq 1$: a contradiction.

\item \emph{Type $1$: there is one boundary lattice point of $\Gamma$ with $Y=1$.} \\
Up to equivalence $\Gamma=\Delta^{(1)}$ is of the form
\begin{center}
    \includegraphics[height=2.5cm]{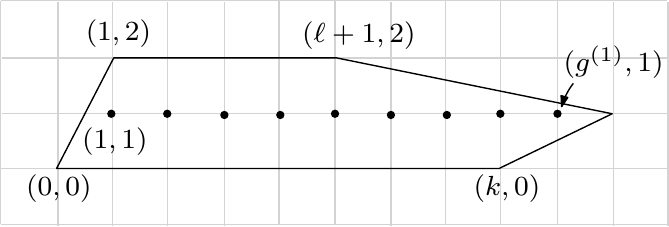}
\end{center}
with $0\leq k\leq 2g^{(1)} + 1$ and
$$\left\{ \begin{array}{l l}
0\leq \ell\leq k & \quad \text{if } \quad 0\leq k\leq g^{(1)},\\
0\leq \ell\leq 2g^{(1)}-k+1 & \quad \text{if } \quad g^{(1)}<k\leq 2g^{(1)} + 1.
\end{array} \right.$$
Since moreover $\Gamma$ is an interior lattice polygon we have that
$\Gamma^{(-1)}$ takes its vertices inside $\mathbb{Z}^2$, leading to the inequalities $k\geq \frac{g^{(1)}-1}{2}$ and $\ell\geq \frac{g^{(1)}-1}{2}$.
For this type, $B=k+\ell-1\geq g^{(1)}-2$ and $B^{(1)}=g^{(1)}-1$.
So if $B\leq B^{(1)}$ then either $k=\ell=\frac{g^{(1)}-1}{2}$ (and $g=4k+4\equiv 0\bmod 4$), or $k=\ell=\frac{g^{(1)}}{2}$ (and $g=4k+3\equiv 3\bmod 4$), or $k=\frac{g^{(1)}+1}{2}$ and $\ell=\frac{g^{(1)}-1}{2}$ (and $g=4k+1\equiv 0\bmod 4$). We find back the polygons $\Gamma_{4k+1},\Gamma_{4k+3},\Gamma_{4k+4}$ from the statement of the lemma.

\item \emph{Type $2$: there are two boundary lattice points of $\Gamma$ with $Y=1$.} \\
Up to equivalence $\Gamma=\Delta^{(1)}$ is of the form
\begin{center}
    \includegraphics[height=2.5cm]{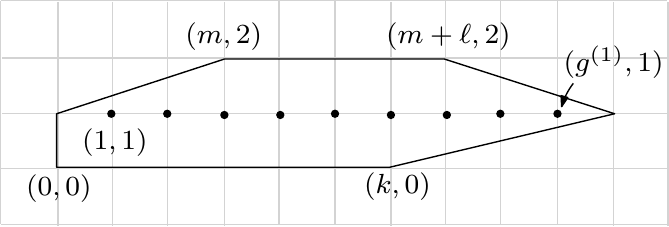}
\end{center}
with $0\leq m\leq g^{(1)}+1$, $0\leq k\leq 2g^{(1)}+2-2m$ and
$$\left\{ \begin{array}{l l}
0\leq \ell\leq k & \quad \text{if } \quad 0\leq k\leq g^{(1)}+1-m,\\
0\leq \ell\leq 2g^{(1)}-k-2m+2 & \quad \text{if } \quad g^{(1)}+1-m<k\leq 2g^{(1)}+2-2m.
\end{array} \right.$$
Since moreover $\Gamma$ is an interior lattice polygon, we also get the inequalities $k\geq \frac{g^{(1)}-1}{2}$ and $\ell\geq \frac{g^{(1)}-1}{2}$. If $B\leq B^{(1)}$ then since $B=k+\ell\geq g^{(1)}-1=B^{(1)}$, we have that $k=\ell=\frac{g^{(1)}-1}{2}$, $B=B^{(1)}=2k$ and $g=4k+5$. So we get the polygons $\Gamma_{4k+5}^m$ from the statement.
\end{itemize}
This concludes the proof.
\end{proof}

\begin{remark}
For each lattice polygon $\Gamma = \Gamma_g, \Gamma_g^m$ appearing in the statement of the lemma, there is only one polygon $\Delta$ for which $\Delta^{(1)}=\Gamma$, namely $\Delta=\Gamma^{(-1)}$. Note that $(\Gamma_4)^{(-1)}\cong 2\Upsilon$ and recall that a $2\Upsilon$-non-degenerate curve is trigonal, rather than tetragonal.
\end{remark}

\section{From toric surfaces to polygons} \label{section_from_surface_to_polygon}

This section is devoted to proving Theorem~\ref{lem_equiv}. As an a priori remark, note
that it is important to impose that $\text{Tor}(\Delta)$ and $\text{Tor}(\Delta')$ are obtained from one another
using a transformation of $\mathbb{P}^N$,
rather than just isomorphic. For instance, let
\[ \Delta = \conv \{ (0,0), (3,0), (3,2), (0,2) \} \quad \ \text{and} \quad \ \Delta' = \conv \{ (0,0), (5,0), (5,1), (0,1) \}, \]
then $\text{Tor}(\Delta), \text{Tor}(\Delta') \subset \mathbb{P}^{11}$ are isomorphic (because their normal fans are the same), but not
projectively equivalent, as they have different degrees ($6$ resp.\ $5$). Here clearly $\Delta \not \cong \Delta'$.

\begin{proof}
We assume familiarity with
the theory of divisors on toric surfaces, along
the lines of \cite[\S3]{CaCo1}. Notation-wise, we will write 
\begin{itemize}
\item $\Sigma_\Delta$ for the (inner) normal fan associated to
a given two-dimensional lattice polygon $\Delta$, and 
\item $\Delta_D$ for the polygon (well-defined up to translation) corresponding
to a Weil divisor (or a Cartier divisor, or an invertible sheaf) $D$ on a given toric surface.
\end{itemize}
The proof then works as follows. Let $\Delta$ and $\Delta'$ be as in the statement of Theorem~\ref{lem_equiv}.
The projective transformation induces an automorphism
$\text{Tor}(\Delta) \rightarrow \text{Tor}(\Delta)$
that sends $\mathcal{O}_{\text{Tor}(\Delta)}(1)$ to $\mathcal{O}_{\text{Tor}(\Delta')}(1)$.
Because
\[ \Delta \cong \Delta_{\mathcal{O}_{\text{Tor}(\Delta)}(1)} \qquad \text{and} \qquad
\Delta' \cong \Delta_{\mathcal{O}_{\text{Tor}(\Delta')}(1)} \]
it suffices to prove the following general statement: if
\[ \iota : \text{Tor}(\Delta) \stackrel{\cong}{\longrightarrow}
\text{Tor}(\Delta') \] is an isomorphism between two toric surfaces, and if $D$ is a Weil divisor
on $\text{Tor}(\Delta)$, then
\[ \Delta_D \cong \Delta_{\iota(D)}.\]
Now it is known that two isomorphic toric varieties always admit a toric
isomorphism between them \cite[Thm.~4.1]{Berchtold}, i.e.\ an isomorphism
that is induced by a $\text{GL}_2(\mathbb{Z})$-transformation taking $\Sigma_{\Delta}$ to $\Sigma_{\Delta'}$. It is clear that such an isomorphism
preserves polygons (up to equivalence). Therefore we may assume that $\Sigma_{\Delta} = \Sigma_{\Delta'}$ and
that $\iota$ is an automorphism of $\text{Tor}(\Delta)$. Every such automorphism can be written as
the composition of
\begin{itemize}
  \item a toric automorphism,
  \item the automorphism induced by the action of an element of $\mathbb{T}^2$,
  \item a number of automorphisms of the form $e_v^\lambda$, where $\lambda \in k$ and $v \in \mathbb{Z}^2$ is a 
  \emph{column vector} of $\Delta$, i.e.\ a
  primitive vector $v$ for which there exists an edge $\tau \subset \Delta$ such that $u + v \in \Delta$
  for all $u \in (\Delta \setminus \tau) \cap \mathbb{Z}^2$. To describe $e_v^\lambda$ explicitly,
  assume that $v = (0,-1)$ and that $\tau$ lies horizontally (the general case can be reduced to this case by using an appropriate unimodular transformation). Then $\text{Tor}(\Delta)$ can be viewed as a compactification
  of $\mathbb{T}^2 \cup (\text{$x$-axis})$ rather than just $\mathbb{T}^2$. On $\mathbb{T}^2 \cup (\text{$x$-axis})$,
  $e_v^\lambda$ acts as $(x,y) \mapsto (x,y + \lambda)$. The column vector property ensures
  that this extends nicely to all of $\text{Tor}(\Delta)$.

  \emph{Example.} Let $\Delta = [0,1] \times [0,1]$ and consider the map
  \[ \varphi : \mathbb{T}^2 \cup (\text{$x$-axis}) \hookrightarrow \text{Tor}({\Delta}) : (x,y) \mapsto (1,x,y,xy). \]
  The point $(x,y + \lambda)$ is mapped to $(1:x:y + \lambda: xy + \lambda x)$. So here
  \[ e_{(0,-1)}^\lambda : (X_{0,0} : X_{1,0} : X_{0,1} : X_{1,1}) \mapsto (X_{0,0} : X_{1,0} : X_{0,1} + \lambda X_{0,0} : X_{1,1} + \lambda X_{1,0}). \]
\end{itemize}
See \cite[Thm.~3.2]{brunsgub} for a proof of this statement, along with a more elaborate discussion.
Now the first type of automorphisms preserves polygons up to equivalence, as before.
The second type also preserves polygons because it preserves torus-invariant Weil divisors.
As for the third type, let $D_\tau$ be the torus-invariant prime divisor
corresponding to the base edge $\tau$ of $v$. Then by adding a divisor
of the form $\text{div}(x^iy^j)$ if needed, one can always find
a torus-invariant Weil divisor that is equivalent to $D$ and whose support does not contain $D_\tau$; see \cite[\S4]{CaCo1} for more details.
But such a divisor is preserved by $e_v^\lambda$, hence the theorem follows.
\end{proof}

\section{Intrinsicness for tetragonal curves} \label{intrinsicness}

We are ready to explain why intrinsicness holds for lattice polygons $\Delta$ satisfying 
\[ \text{lw}(\Delta^{(1)}) = 2 \quad \ \text{and} \quad \ B \geq B^{(1)} + 2, \]
that is, for the polygons of type (e) from the introduction. Let $C$ be a $\Delta$-non-degenerate curve. Then it is
a tetragonal curve (indeed, $B \geq B^{(1)} + 2$ implies $\Delta \not \cong 2\Upsilon$) whose Schreyer invariants
$b_1, b_2$ satisfy $b_1 \geq b_2 + 2$. By Theorem~\ref{theorem_schreyer} we find that Schreyer's surface $\mu(Y) \subset \mathbb{P}^{g-1}$
equals $\text{Tor}(\Delta^{(1)})$. Now 
suppose that $C$ is also $\Delta'$-non-degene\-rate for some two-dimensional lattice polygon
$\Delta'$. By the tetragonality of $C$ we have $\text{lw}(\Delta'^{(1)}) = 2$.
In analogy with the previous notation, write
\[
B' = \sharp(\partial \Delta'^{(1)}\cap\mathbb{Z}^2)-4, \qquad
B'^{(1)} = \sharp(\Delta'^{(2)}\cap\mathbb{Z}^2)-1, \]
so that $\left\{ B', B'^{(1)} \right\} = \{b_1,b_2\}$ by Theorem~\ref{theorem_schreyer}. It follows that either
\[ B' \geq B'^{(1)} + 2 \qquad \text{or} \qquad B'^{(1)} \geq B' + 2. \]
But the latter is impossible by Lemma~\ref{lemma_comb}, which states that $B'^{(1)}$ is at most $B' + 1$. Therefore
$B' > B'^{(1)}$ and, again by Theorem~\ref{theorem_schreyer}, we find that $\mu(Y)$ is given by
$\text{Tor}(\Delta'^{(1)})$. We conclude that $\text{Tor}(\Delta^{(1)})$ and $\text{Tor}(\Delta'^{(1)})$
are equal, possibly modulo a projective transformation. 
Intrinsicness now follows from Theorem~\ref{lem_equiv}.\\

\noindent This argument can be refined. For instance, in genus $g \not \equiv 0 \bmod 4$ it suffices that
$B \geq B^{(1)} + 1$, because
in this case Lemma~\ref{lemma_comb} yields the sharper bound $B'^{(1)} \leq B'$. In genus $g \equiv 2 \bmod 4$ one sees that this
is automatically satisfied.\\ 

\noindent By pushing this type of reasoning, we obtain the following statement.

\begin{theorem} \label{thm_intrinsicness}
Let $\Delta, \Delta'$ be two-dimensional lattice polygons and let there be a curve that is both $\Delta$-non-degenerate
and $\Delta'$-non-degenerate. Suppose that $\emph{lw}(\Delta^{(1)}) = 2$ and define $g = \sharp (\Delta^{(1)} \cap \mathbb{Z}^2) = \sharp (\Delta'^{(1)} \cap \mathbb{Z}^2)$. 
%Assume that $g \geq 4$.
\begin{itemize}
  \item Case $g \equiv 0 \bmod 4$. If $\Delta^{(1)}, \Delta'^{(1)} \not \cong \Gamma_g$ then $\Delta^{(1)} \cong \Delta'^{(1)}$.
  This is automatic if $\sharp (\partial \Delta^{(1)} \cap \mathbb{Z}^2) \geq \sharp (\Delta^{(2)} \cap \mathbb{Z}^2) + 5$.  
  \item Case $g \equiv 1 \bmod 4$. If $\Delta^{(1)}, \Delta'^{(1)} \not \cong \Gamma_g^m$ for all $1 \leq m \leq (g+3)/4$ then
   $\Delta^{(1)} \cong \Delta'^{(1)}$. This is automatic if $\sharp (\partial \Delta^{(1)} \cap \mathbb{Z}^2) \geq \sharp (\Delta^{(2)} \cap \mathbb{Z}^2) + 4$.
  \item Cases $g \equiv 2, 3 \bmod 4$. Here one always has $\Delta^{(1)} \cong \Delta'^{(1)}$.
\end{itemize}
\end{theorem}

\begin{proof}
The cases $g \equiv 0,2 \bmod 4$ follow along the above lines of thought. For the case $g \equiv 3 \bmod 4$ one remarks
that Schreyer's invariants coincide if and only if $B = B^{(1)}$, which by Lemma~\ref{lemma_comb} happens if and only
if $\Delta^{(1)} \cong \Delta'^{(1)} \cong \Gamma_g$. If not then $B \geq B^{(1)} + 1$, and one proceeds as before.

The most subtle case is when $g \equiv 1 \bmod 4$. If $g = 5$ then Schreyer's invariants coincide if and only if 
$\Delta^{(1)} \cong \Delta'^{(1)} \cong \Gamma^0_5$ (indeed, the polygons $\Gamma^1_5$ and
$\Gamma^2_5$ appearing in Lemma~\ref{lemma_comb} were excluded in the statement), so this is analogous to the $g \equiv 3 \bmod 4$ case.
If $g > 5$ then one draws the weaker conclusion that
Schreyer's invariants coincide if and only if
$\Delta^{(1)}$ and $\Delta'^{(1)}$ are among $\Gamma_g$ and $\Gamma^0_g$. To distinguish between both cases, one notes that 
the scrollar invariants $e_1,e_2, e_3 $ are
\[ \frac{g-5}{4}, \frac{g-1}{4}, \frac{g-3}{2} \qquad \text{and} \qquad \frac{g-5}{4}, \frac{g-5}{4}, \frac{g-1}{2}, \]
respectively. Here we implicitly used that our curve carries a unique $g^1_4$ by \cite[Cor.\,6.3]{CaCo1}, so it
does make sense to talk about \emph{the} scrollar invariants. We conclude that 
$\Delta^{(1)} \cong \Delta'^{(1)} \cong \Gamma^0_g$ if the curve has two coinciding scrollar invariants, 
and that $\Delta^{(1)} \cong \Delta'^{(1)} \cong \Gamma_g$ if not.
\end{proof}

\begin{comment}
\begin{proof}
Now assume that a tetragonal curve $C$ is $\Delta_1$- and $\Delta_2$-non-degenerate, and that both $\Delta_1^{(1)}$ and $\Delta_2^{(1)}$ are not among the polygons $\Gamma_{4k+5}^m$ and $\Gamma_{4k+4}$.
Consider Schreyer's tetragonal invariants $b_1,b_2$ of $C$ (with $b_1\geq b_2$) and the corresponding surfaces $Y,Z\subset\mathbb{P}(\mathcal{E})$ containing $C'$. Write $B_i=\sharp(\partial\Delta_i^{(1)}\cap\mathbb{Z}^2)-4 $ and $B^{(1)}_i=\sharp(\Delta_i^{(2)}\cap\mathbb{Z}^2)-1$ for $i=1,2$. Because of Lemma \ref{lemma_comb} and the assumptions, $B_1\geq B^{(1)}_1$ and $B_2\geq B^{(1)}_2$, hence Theorem \ref{theorem_schreyer} implies that $b_1=B_1=B_2$ and $b_2=B^{(1)}_1=B^{(1)}_2$. If $b_1>b_2$, we get $\mu(Y)=\text{Tor}(\Delta_1^{(1)})=\text{Tor}(\Delta_2^{(1)})$ and the statement follows by applying Lemma \ref{lem_equiv}). If $b_1=b_2$, then $\Delta_1$ and $\Delta_2$ are both mentioned in the list of Lemma \ref{lemma_comb}. We immediately get that $\Delta_1^{(1)}\cong \Delta_2^{(1)}\cong \Gamma_g$ using the assumptions.
\end{proof}
\end{comment}

\begin{remark}
%The condition $g \geq 4$ is not essential: intrinsicness is trivially satisfied for $g \leq 3$ (e.g.\ this is covered
%by cases (a) and (b) from the introduction). Also 
Note that the theorem remains valid
if we replace `for all $1 \leq m \leq (g+3)/4$' by `for all $m \in \{0, \dots, (g+3)/4 \} \setminus \{i\}$', for whatever
$i$.
\end{remark}

\begin{example}
Let $g\geq 4$ satisfy $g\equiv 0\bmod 4$, and denote by $\Delta_g$ the (unique) lattice polygon
for which $\Delta_g^{(1)} = \Gamma_g$. Then it is possible that a 
$\Delta_g$-non-degenerate curve is also non-degenerate with respect to a lattice polygon $\Delta'$ for which 
$\Delta'^{(1)}\not\cong \Gamma_g$. For instance, consider 
$f=1-x^2y^4-x^{\frac{g}{2}+2}y^2$ and $f'=(y^4-1)x^{\frac{g}{2}+1}+4y^2$. Both polynomials
are non-degenerate with respect to their respective Newton polygons. Note that $\Delta(f) \cong \Delta_g$ and
that $\Delta(f')^{(1)} \not \cong \Gamma_g$. Now the rational maps
$$U_f \to U_{f'}:(x,y)\mapsto\left(x,\frac{1-xy^2}{x^{\frac{g}{4}+1}y}\right)$$
$$U_{f'} \to U_f:(x,y)\mapsto\left(x,\frac{2y}{x^{\frac{g}{4}+1}(1+y^2)}\right)$$ are inverses of each other, 
so $C_f$ and $C_{f'}$ are isomorphic. We conclude 
that $C_f$ is both $\Delta_g$-non-degenerate and $\Delta(f')$-non-degenerate.
\manuallabel{example0mod4}{\thetheorem}
\end{example}

%\begin{remark}
%If $g\equiv 1\bmod 4$, a curve $C$ cannot be non-degenerate with respect to both the polygon $(\Gamma_g)^{(-1)}$ and a polygon of the form $(\Gamma_g^m)^{(-1)}$ since the scrollar invariants are different ...
%\end{remark}

\begin{example}[Example]
%\textcolor[rgb]{1,0,0}{TO DO}: family $\Gamma_g^m$ in genus $g=5$ or $9$. \\
We conjecture that for 
each $g\geq 5$ with $g\equiv 1\bmod 4$ and each $0\leq n,m\leq (g+3)/4$, there exists a curve that is both $\Delta_g^n$- and $\Delta_g^m$-non-degenerate. 
Here $\Delta_g^n$ and $\Delta_g^m$ are the unique lattice polygons having $\Gamma_g^n$ and $\Gamma_g^m$
as their respective interiors. 

Loosely speaking, we believe that the following strategy for finding such a curve always works (although we could not prove this). 
From Sections~\ref{introsection} and~\ref{sect_schreyersinv} we know that the 
canonical model $C_f^\text{can}$ of a $\Delta^n_g$-non-degenerate curve $C_f$ satisfies
$C_f^\text{can} \subset \text{Tor}(\Gamma_g^n) \subset S \subset \mathbb{P}^{g-1}$, where $S$
is a rational normal scroll of type
\[ \left( \frac{g-5}{4}, \frac{g-5}{4}, \frac{g-1}{2} \right), \]
and that $C_f^\text{can}$ arises as the intersection of two surfaces $Y$ and $Z$ 
inside the class 
\[ 2H - \frac{g-5}{2} R \]
(the role of $\mu$, which is only relevant for $g = 5$, is ignored for the sake of exposition).
Recall from Remark~\ref{rmk_eqYZ} that 
one can take $Y = \text{Tor}(\Gamma_g^n)$, and $Z = \mathcal{F}_f$. 
The idea is to switch the role of $Y$ and $Z$, in the sense 
that one chooses $f$ such that 
$\mathcal{F}_f = \theta(\text{Tor}(\Gamma^m_g))$ for some $\theta \in \text{Aut}(S) \subset \text{Aut}(\mathbb{P}^{g-1})$.
Because non-degeneracy is generically satisfied, one expects $\theta^{-1}(Y)$ to be of the form $\mathcal{F}_{f'}$
for some $\Delta_g^m$-non-degenerate Laurent polynomial $f'$.

Explicit examples in genus $g = 5$ can be found in our unpublished \texttt{arXiv} paper
\texttt{1304.4997}. For $g=9$ and $\{n,m\}=\{0,3\}$
\begin{figure}[h!]
\centering
     \includegraphics[height=2.5cm]{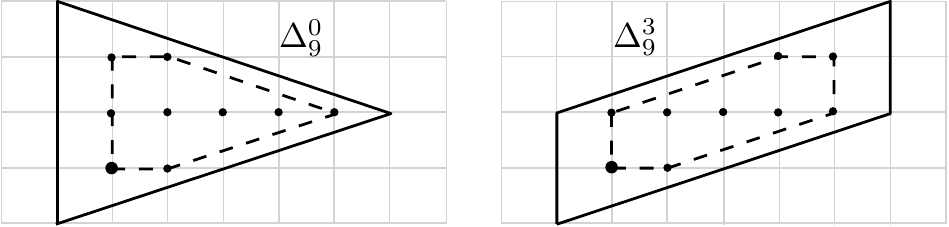}
\end{figure}
we used the above approach to find that the $\Delta^0_9$-non-degenerate Laurent polynomial
\[ \Scale[0.78]{f =8x^5y + 36x^4y + 66x^3y - x^2y^2 + 62x^2y - x^2 + 33xy + 9y - 2x^{-1}y^3 - 2x^{-1}y^2 - 4x^{-1}y - 3x^{-1} - 3x^{-1}y^{-1}}\]
and the $\Delta^3_9$-non-degenerate Laurent polynomial
\[ \Scale[0.78]{f' =2x^5y^3 + x^5y^2 - x^5y - 6x^4y - 15x^3y + 2x^2y^2 - 14x^2y + x^2 - 15xy - 6y - x^{-1}y + 3x^{-1} + 3x^{-1}y^{-1}} \]
define birationally equivalent curves in $\mathbb{T}^2$. To describe
the automorphism $\theta$ explicitly, we need to pick coordinates of $\mathbb{P}^{g-1}$.
When thought of as the ambient space of $\text{Tor}(\Gamma_9^0)$, we will write
\[ \mathbb{P}^{g-1} = \text{Proj}\, V \qquad \text{with $V = k[X_{0,0},X_{1,0},X_{0,1},X_{1,1},X_{2,1},X_{3,1},X_{4,1},X_{0,2},X_{1,2}]$}, \]
where $X_{i,j}$ is the coordinate corresponding to the lattice point $(i,j) \in \Gamma_9^0$ (the origin is understood to
be the bold-marked lattice point). Similarly, when thought of as the ambient space of $\text{Tor}(\Gamma_9^3)$
we write
\[ \mathbb{P}^{g-1} = \text{Proj}\, W \qquad \text{with $W = k[X_{0,0},X_{1,0},X_{0,1},X_{1,1},X_{2,1},X_{3,1},X_{4,1},X_{3,2},X_{4,2}]$}. \]
Then, on the level of coordinate rings, $\theta : V\to W$ can be defined by
\footnotesize $$\begin{pmatrix}\theta(X_{0,1})\\ \theta(X_{1,1})\\ \theta(X_{2,1})\\ \theta(X_{3,1})\\ \theta(X_{4,1})\\ \theta(X_{0,2})\\ \theta(X_{1,2})\\ \theta(X_{0,0})\\ \theta(X_{1,0})\end{pmatrix}=\begin{pmatrix}1&4&6&4&1&0&0&0&0 \\ 1&5&9&7&2&0&0&0&0 \\ 1&6&13&12&4&0&0&0&0 \\
1&7&18&20&8&0&0&0&0 \\ 1&8&24&32&16&0&0&0&0 \\ 1&1&0&0&0&1&1&1&1 \\ 1&2&0&0&0&1&2&1&2 \\ 0&0&0&1&1&2&2&3&3 \\ 0&0&0&1&2&2&4&3&6 \end{pmatrix}\cdot \begin{pmatrix}X_{0,1}\\ X_{1,1}\\ X_{2,1}\\ X_{3,1}\\ X_{4,1}\\ X_{3,2}\\ X_{4,2}\\ X_{0,0}\\ X_{1,0}\end{pmatrix}.$$ \normalsize
We leave it to the reader to verify that $\theta$ maps $S$ to $S$ and sends $\text{Tor}(\Gamma_9^3)$ to $\mathcal{F}_f$ and $\mathcal{F}_{f'}$ to $\text{Tor}(\Gamma_9^0)$ (for an appropriate
choice of defining equations for $\mathcal{F}_f$ and $\mathcal{F}_{f'}$).
\begin{comment}
This map satisfies $\theta(\mathcal{I}(C_1))=\mathcal{I}(C_2)$, where %$\mathcal{I}(C_i)$ is the ideal of the curve $C_i$ for $i\in\{1,2\}$,
$C_1\subset \mathbb{P}^8=\text{Proj}(V)$ is the canonical model corresponding to
and $C_2\subset \mathbb{P}^8=\text{Proj}(W)$ is the canonical model corresponding to
Hence, the $\Delta_9^0$-non-degenerate curve $U(f_1)$ is birationally equivalent to the $\Delta_9^3$-non-degenerate curve $U(f_2)$.
In fact, we can relate what happens here to Remark \ref{rmk_eqYZ} : using the isomorphisms $\mathbb{P}(\mathcal{E})\cong S=S(1,1,4)$ in both cases, we have that $\theta(\mathcal{I}(S))=\mathcal{I}(S)$, $\theta(\mathcal{I}(Y_{f_1}))=\mathcal{I}(Z_{f_2})$ and $\theta(\mathcal{I}(Z_{f_1}))=\mathcal{I}(Y_{f_2})$.
\end{comment}
\end{example}

\noindent \textsc{Vakgroep Wiskunde, Universiteit Gent}\\
\noindent \textsc{Krijgslaan 281, 9000 Gent, Belgium}\\
\noindent \emph{E-mail address:} \verb"wouter.castryck@gmail.com"\\

\noindent \textsc{Dept.\ of Mathematics and Applied Mathematics, University of Cape Town}\\
\noindent \textsc{Private Bag X1, Rondebosch 7701, South Africa}\\
\noindent \emph{E-mail address:} \verb"filip.cools@uct.ac.za"\\

\end{document}